\documentclass[11pt,reqno]{amsart}

\usepackage{graphicx,amscd,amsmath,amsfonts,amssymb} 
\usepackage{amssymb}
\usepackage{amsmath}
\usepackage{amsfonts}
\usepackage{graphicx}
\usepackage[initials]{amsrefs}
\usepackage{enumerate}

\newtheorem{theorem}{Theorem}[section]
\newtheorem{lemma}[theorem]{Lemma}
\newtheorem{proposition}[theorem]{Proposition}

\newtheorem{definition}{Definition}[section]

{\theoremstyle{definition}
\newtheorem{remark}[theorem]{Remark}

}
\newtheorem{problem}{Problem}

\def\N{\mathbb N}

\def\R{\mathbb R}
\def\Q{\mathbb Q}
\def\dis{\displaystyle}
\def\eps{\varepsilon}

\DeclareMathOperator*{\esssup}{ess\,sup}

\title{Lineability and modes of convergence}

\author {M.C.~Calder\'on-Moreno, P.J.~Gerlach-Mena, J.A.~Prado-Bassas}


\thanks{The authors
have been partially supported by the Plan Andaluz de Investigaci\'on de la Junta de Andaluc\'{\i}a FQM-127, Grant P08-FQM-03543, and by MICINN Grant PGC2018-098474-B-C21.}
\date{}

\begin{document}

\maketitle

\begin{abstract}
\noindent In this paper we look for the existence of large linear and algebraic structures of sequences of measurable functions with different modes of convergence. Concretely, the algebraic size of the family of sequences that are convergent in measure but not a.e.~pointwise, uniformly but not pointwise convergent, and uniformly convergent but not in $L^1$-norm, are analyzed. These findings extend and complement a number of earlier results by several authors.

\vskip .25cm


\end{abstract}

\bigskip

\begin{center}{ \em Dedicated to Professor Bernal-Gonz\'alez in occasion of his 60th birthday.}
\end{center}

\bigskip

\section{Introduction}

The notion of convergence is natural and clear in the scalar case or even in the finite dimensional setting. But when we speak about the limit of a sequence of functions, it is crucial to define what mode of convergence we are going to deal with. Every undergraduate student knows the notions of pointwise and uniform convergence of a sequence of functions, but if we endow our space of functions with a measure or even a norm, we can speak about convergence in measure, convergence almost everywhere, almost uniformly convergence, convergence in norm, etc. In this paper we investigate the differences between some of these modes of convergence in sequence spaces of measurable functions from the point of view of lineability.

The notion of lineability was first coined by Gurariy in 1966 (see \cite{Gurariy1}) in order to look for linearity into nonlinear settings. Later, in 2009, this notion was extended by Aron, Seoane, Gurariy {\em et al.}~to an algebraic frame (see \cite{topology,AGS}). Let us introduce the definitions of this new branch.

\begin{definition}
Given a topological vector space $X$, a subset $A \subset X$ and a cardinal number $\alpha$, we say that $A$ is:
\begin{itemize}
\item {\em lineable} if there is an infinite dimensional vector space $M$ such that $M \setminus \{0\} \subset A$,
\item {\em $\alpha$-lineable} if there is an a vector space $M$ with $\text{dim}(M) = \alpha$ such that $M \setminus \{0\} \subset A$ (hence lineability means $\aleph_0$-lineability, where $\aleph_0 = {\rm card}\,(\N )$ and $\N$ denotes the set of positive integers).
\item {\em maximal lineable} if $A$ is $\text{dim}(X)$-lineable.
\end{itemize}
Moreover, if $X$ is a topological vector space, then $A$ is said to be:
\begin{itemize}
\item {\em spaceable} in $X$ when there exists a closed infinite dimensional vector space $M$ such that $M \setminus \{0\} \subset A$;
\item {\em dense-lineable} in $X$ whenever there is a dense vector subspace $M$ with $M \setminus \{0\} \subset A$;
\item {\em $\alpha$-dense-lineable} in $X$ when, in addition, $\text{dim}(M)=\alpha$;
\item {\em maximal dense-lineable} in $X$ if we also have that $\text{dim}(M) = \text{dim}(X)$.
\end{itemize}
When $X$ is a topological vector space contained in some (linear) algebra, then $A$ is called:
\begin{itemize}
\item {\em algebrable} if there is an algebra $M$ so that $M \setminus \{0\} \subset A$ and $M$ is infinitely generated, that is, the cardinality of any system of generators of $M$ is infinite;
\item {\em $\alpha$-algebrable} if  there is an $\alpha$-generated algebra $M$ with $M \setminus \{0\} \subset A$.
\item {\em strongly $\alpha$-algebrable} if, in addition, the algebra $M$ can be taken free.
\end{itemize}
\end{definition}

If the algebraic structure of $X$ is commutative, the strong algebrability is equivalent to the existence of a generating set $B$ of the algebra $M$ such that, for any $s \in \N$, any nonzero polynomial $P$ in $s$ variables without constant term and any distinct $f_1,\ldots,f_s \in B$, we have $P(f_1, \ldots ,f_s) \in M\setminus\{0\}$.

\bigskip

There is a wide literature about lineability and we refer the reader to the survey \cite{BAMS} and the book \cite{aronbernalpellegrinoseoane} for a huge compendium of results in many situations. But, up to the knowledge of the authors, the study of lineability in function sequence spaces has not been extensively treated, probably because of the difficulty to define an appropriate topology to handle with.

The first result about lineability in sequence spaces goes back to 2014, when Bernal and Ord\'o\~nez \cite{bernalordonez2014} proved, as a consequence of a general result, the spaceability and maximal lineability of the family of sequences $(f_n:\R\to\R)_n$ of continuous bounded and integrable functions such that $\|f_n\|_\infty\overset{n\to\infty}{\longrightarrow}{0}$, $\sup\{\|f_n\|_{1}:\,n\in\N\}<+\infty$ but $\|f_n\|_{1}\not\to 0$ ($n\to\infty$). Here, as usual, $\R$ will denote the real line and for $g:\R\to\R$, $\|g\|_\infty$ will stand for the supremum of $|g|$ over $\R$ and $\|g\|_1$ for the integral of $|g|$ (respect to the classical Lebesgue measure) over $\R$.

Later, in 2017, Ara\'ujo, {\em et al.}~\cite{ABMPS-studia} showed the $\mathfrak{c}$-lineability (where $\mathfrak{c}$ denotes the cardinality of the continuum) of the family of sequences of Lebesgue measurable functions $\R\to\R$ such that $f_n$ converges pointwise to zero and $f_n(I)=\R$ for any non-degenerate interval $I\subset\R$ and any $n\in\N$, as well as the maximal dense-lineability (in the vector space of sequences of Lebesgue-measurable functions $[0,1]\to\R$) of the family of all sequences of Lebesgue-measurable functions such that $f_n\to0$ in measure but not almost everywhere in $[0,1]$.

In the probability theory setting, Conejero {\em et al.}~(see \cite{CoFeMuSe}) studied in 2017 some lineability and algebrability problems, as for example, convergent not $L^1$-unbounded martingales, pointwise convergent random variables whose means do no not converge to the expected value, stochastic processes that are $L^2$ bounded and convergent but not pointwise convergent in a null set.

Finally, in 2019, the authors \cite{CGP_JMAA} showed  that the family of sequences of unbounded continuous and integrable functions in $[0,+\infty)$ converging to zero uniformly in compacta and in $L^1$-norm is maximal dense-lineable and strongly algebrable. Moreover, they also prove that the uniform convergence in compacta of $[0,+\infty)$ can be strengthened to almost uniformly convergence, but not to uniform convergence in $[0,+\infty)$.

A number of results about Banach or quasi-Banach spaces of vector valued sequences, but not directly related to modes of convergence, can be found in \cite{BaBoFaPe,BoDiFaPe, BoFa, NoPel}.

In this paper we try to contribute to the study of the algebraic structure inside function sequence spaces. Concretely, in Section 2 we are interested in the algebraic structure of sequences converging to zero in measure but not almost everywhere (a.e.), thus extending results of Ara\'ujo {\em et al.}~\cite{ABMPS-studia}. Section 3 is devoted to analyze the difference between uniform, pointwise, uniformly a.e.~and almost uniform convergence; in particular we show the strong-$\mathfrak{c}$-algebrability, spaceability and maximal dense-lineability of the family of all sequences of functions converging to zero pointwise a.e.~and almost uniformly but not uniformly in $[0,1]$. Finally, in Section 4, the algebraic size of the set of sequences of functions that are uniformly convergent to zero in $[0,+\infty)$ but not in $L^1$-norm will be stated.

\section{Measure vs Pointwise a.e.~convergence}


Let $\mathcal{L}_0$ be the set of all measurable (with respect to the Lebesgue measure $m$) functions $[0,1] \to \R$. We denote by $L_0$ the vector space of all classes of functions in $\mathcal{L}_0$, where two functions are identified whenever they are equal a.e.~in $[0,1]$. In $L_0$, it is natural to consider the next two types of convergence: the (pointwise) a.e.~convergence and the convergence in measure. Recall that a sequence $(f_n)_n \subset L_0$ converges to $f$ a.e.~in $[0,1]$ if there is a measurable set $E$ such that $m(E)=0$ and $f_n \to f$ pointwise in $[0,1] \setminus E$ ($n \to \infty$); and $(f_n)_n \subset L_0$ converges to $f$ in measure if
$$\lim_{n\to \infty} m (\{ x \in [0,1]: \ |f_n(x)-f(x)| > \varepsilon \} )=0 \quad {\rm for} \ {\rm all} \ \varepsilon >0.$$

Using Egorov's Theorem (see, for instance, \cite[Theorem 8.3]{Oxtoby}), it is clear that convergence a.e.~in $[0,1]$ (in fact, in any finite-measure set) is stronger than convergence in measure, but they are not equivalent. For instance, the so-called ``Typewriter sequence'' given by
\begin{equation}\label{typewritter}T_n:= \chi_{[j 2^{-k}, (j+1) 2^{-k}]}\end{equation}
(where, for each $n$, the non-negative integers $j$ and $k$ are uniquely determined by $n=2^k+j$ and $0 \leq j < 2^k$, and $\chi_A$ denotes the characteristic function of $A$) satisfies that $T_n \to 0 $ ($n \to \infty$) in measure (moreover, $\lim_{n\to\infty}m(\{x\in[0,1]:\, T_n(x)\ne0\})=0$) but, for every point $x_0 \in [0,1]$ the sequence $(T_n(x_0) )_n$ does not converge, since it takes infinitely many times the value 0 and infinitely many times the value 1. However, the difference between both types of convergence is not so big. In fact, convergence in measure of a sequence $(f_n)_n$ to $f$ implies a.e.~convergence to $f$ of some subsequence (see, for instance, \cite[Theorem 21.9]{26-studia}).

In the following lemma we define a metric in $L_0$ that will be used later (see \cite{26-studia}).
\begin{lemma}\label{metrica} Let be $\rho : L_0 \times L_0 \to [0,+\infty)$  the metric given by
$$\rho(f,g) = \int_{[0,1]} \frac{ |f(x) - g(x)| }{ 1 + |f(x) - g(x)|} dx.$$
Then $\rho(f_n,f) \to 0$ if and only if $f_n \to f$ in measure.
\end{lemma}
Observe that under the topology generated by $\rho$, the space $L_0$ turns out to be a separable complete metrizable topological vector space, and the space $L_0^\N$, endowed with the natural product metric
$$D((f_n)_n,(g_n)_n)=\sum_{n=1}^\infty \frac{1}{2^n}\cdot \frac{\rho(f_n,g_n)}{1+\rho(f_n,g_n)},$$
becomes also a complete metrizable separable topological vector space.
If we ask for the amount of sequences with the same behaviour than the typewriter sequence, we can say that, in 2017, the authors of \cite{ABMPS-studia} proved the maximal dense-lineability in $L_0^{\N}$ of the family of sequences $(f_n)_n \in L_0^{\N}$ such that $f_n \to 0$ in measure but $(f_n)_n$ does not converges a.e.~in $[0,1]$. In the next result we show that this family is also large in an algebraic sense. Let us represent each $N$-tuple $(r_1, \dots ,r_N) \in \R^N$ by ${\bf r}$, and set
$|{\bf r}| := r_1 + \cdots + r_N$ and ${\bf r} \cdot {\bf s} := r_1s_1 + \cdots + r_Ns_N$.

\begin{theorem}\label{Teo:1} The family of classes of sequences $(f_n)_n \in L_0^\N$ such that $f_n \to 0$ in measure but $(f_n)_n$ does not converge (to zero) a.e.~in $[0,1]$ is strongly $\mathfrak{c}$-algebrable.
\end{theorem}

\begin{proof}
Let \,$H \subset (0,+\infty )$ be a linearly \,$\Q$-independent set (where $\Q$ will denote the set of all rational numbers) with
${\rm card} \, (H)= \mathfrak{c}$. For each \,$c \in H$, we define the sequence $F(c)=(F(c,n))_n$ by
\begin{equation*}F(c,n)(x):= e^{-cx} \cdot T_n (x),\end{equation*}
where $(T_n)_n$ is the Typewriter sequence defined in \eqref{typewritter}.

Let $\mathcal{B} $ be the algebra generated by the family of sequences $\{ F(c):\, c \in H\}$, that is, $\mathcal{B} $ is the family of all sequences $(F_n)_n$ for which there exist $N \in \N$, mutually different $c_1, \dots ,c_N \in H$ and a nonzero polynomial $P$ in $N$ real variables without constant term such that $F_n = P(F({c_1},n), \dots , F({c_N},n))$ for every $n\in \N$. Therefore, there exist a nonempty finite set $J \subset \N_0^N \setminus \{(0,\overset{(N)}{\ldots} ,0)\}$ and scalars $\alpha_{\bf j} \in \R \setminus \{0\}$, for ${\bf j}\in J$, such that, for all $x\in\R$,
\begin{eqnarray} \label{Eq:1}
F_n(x)&=&\dis 
\sum_{{\bf j} \in J} \alpha_{\bf j} \, F(c_1,n)(x)^{j_1} \cdots F(c_N,n)(x)^{j_N} \nonumber\\
&=&\dis \sum_{{\bf j}\in J} \alpha_{\bf j} e^{-({\bf c}\cdot{\bf j}) x}T_n(x) ^{|\bf j|} = \left(\sum_{{\bf j}\in J} \alpha_{\bf j} e^{-({\bf c}\cdot{\bf j}) x}\right)T_n(x),\end{eqnarray} where in the last equality the fact that $T_n(x)$ is an indicator function (so $T_n(x)^\beta=T_n(x)$ for any $\beta>0$)  is crucial.

Put $\varphi_{{\bf c},J}(x):= \sum_{{\bf j}\in J} \alpha_{\bf j} e^{-({\bf c}\cdot{\bf j}) x}$. Because $H$ is a $\Q$-linearly independent set, all numbers $\{{\bf c}\cdot {\bf j}:\, {\bf j}\in J\}$ are mutually distinct, so $\{e^{-({\bf c}\cdot{\bf j})x}:\, {\bf j}\in J\}$ is a linearly independent set of functions and, by \eqref{Eq:1}, $F_n$ is always non-null and the algebra ${\mathcal B}$ is a free algebra.

Moreover, the sequence $(F_n)_n=( \varphi_{{{\bf c},J}}(x) T_n(x))_n$ converges to zero in measure because, for any $\varepsilon>0$,
$$\{ x \in [0,1]: \, |\varphi_{{\bf c},J}(x) T_n(x)| > \varepsilon \} \subset \{ x \in [0,1] : \, T_n(x)\ne0\},$$ and the measures of the lasts sets go to zero when $n\to\infty$, as pointed out at the beginning of the section. Finally, as $(T_n)_n$ does not converge to $0$ pointwise a.e.~in $[0,1]$ and $\varphi_{{\bf c},J}(x)=0$ for finitely many points (specifically, at most ${\rm card}(J)-1$ points), \eqref{Eq:1} also gives us that any non-zero member of $\mathcal{B}$ does not converge to $0$ pointwise a.e.~in $[0,1]$.
%
\end{proof}

If we take into account the topological structure of $L_0^\N$, we can also extend the mentioned result by Ara\'ujo {\em et al.}~to get spaceability for the set of sequences of functions converging to zero in measure but not pointwise convergent a.e.~in $[0,1]$.

\begin{theorem}\label{spaceability1}The family of classes of sequences $(f_n)_n \in L_0^\N$ such that $f_n \to 0$ in measure but $(f_n)_n$ does not converge (to zero) a.e.~in $[0,1]$ is spaceable in $L_0^\N$.
\end{theorem}

\begin{proof} Firstly, let us divide $\N$ into infinitely many strictly increasing pairwise disjoint subsequences $\{(i(k,n))_n:\, k\in\N\}$ (for instance, $i(k,n)=\frac{k(k+1)}{2}+(n-1)k$). For each $k\in\N$, define the sequence $T(k)=(T(k,n))_n$ as follows: $$T(k,n)=\begin{cases}
  \chi_{[j 2^{-i(k,m)}, (j+1) 2^{-i(k,m)}]} & \text{if } n=j+2^{i(k,m)},\ 0\le j< 2^{i(k,m)}\\
  0 & \text{elsewhere.}
\end{cases} $$
Roughly speaking, for fixed $k\in\N$, we preserve every term of the Typewriter sequence where the support has length $2^{-i(k,m)}$ ($m\in\N$) and change the rest to be 0. Similarly as the Typewriter sequence, it is straightforward that every sequence $T(k)$ is convergent to zero in measure. Moreover, by construction, given any $x\in [0,1]$, there are infinitely many terms of $T(k)$ where the sequence takes the value 1 and infinitely many terms where it takes the value 0; so the sequence $(T(k,n)(x))_n$ cannot be convergent (to zero).

By construction, if $(k,n)\ne (k',n')$, then $T(k,n)$ and $T(k',n')$ can not be both nonzero. So if we take a linear combination $$\lambda_1T(k_1)+\ldots+\lambda_NT(k_N)$$ for any $1\le j_0\le N$ we always can find $n_{j_0}\in\N$ such that $T(k_{j_0},n_{j_0})$ is nonzero but $T(k_j,n_{{j_0}})=0$ for $j\ne j_0$. So, writing the linear combination at the $n_{j_0}$-coordinate, we get that $\lambda_{j_0}T(k_{j_0},{n_{j_0}})=0$, whence $\lambda_{j_0}=0$ and we get that the set $\{T(k):\, k\in\N\}$ is linearly independent.

Now let $M:=\overline{\rm span}\{T(k):\, k\in\N\}$. It is clear that $M$ is a closed infinite dimensional subspace of $L_0^\N$. In addition, it is straightforward that every member of $M$ is a sequence converging to zero in mesure (in fact,~using a standard topological argument it is easy to show that the whole vector space of sequences converging to zero in measure is closed in $L_0^\N$). 

It remains to prove the non pointwise convergence to zero. For that purpose, observe that
every nonzero member of $M$ is a finite or infinite linear combination of sequences $T(k)$, to be more precise, since the interiors of the supports of all functions in $\{T(k):\,k\in\N\}$ are pairwise disjoint, if $F\in M\setminus\{0\}$, there exists $J\subset\N$ and $\lambda_\nu\in \R\setminus\{0\}$ for every $\nu\in J$, such that $F=\sum_{\nu\in J}\lambda_j T(k_\nu)$. Fix $\nu_0\in J$ and let $J_0:=\{2^{i(k_{\nu_0},m)}+j:\, m\in\N,\, 0\le j< 2^{i(k_{\nu_0},m)}\}$. By construction, $F_n=\lambda_{\nu_0}T(k_{\nu_0},n) $ for every $n\in J_0$; hence, for fixed $x\in[0,1]$ there are infinitely many natural numbers $n$ (at least every number in $J_0$) such that $F_n(x)=\lambda_{\nu_0}\ne0$ and infinitely many natural numbers $n$ such that $F_n(x)=0$. Hence $(F_n(x))_n$ is never convergent (so $(F_n)$ s not a.e.~convergent) to zero.
\end{proof}

\section{Pointwise vs Uniformly convergence}

In this section we concentrate on pointwise convergence and uniform convergence in $[0,1]$. It is an easy exercise to find examples of sequences which converges to zero pointwise but not uniformly:~take, for instance, the sequence
\begin{equation}\label{ejemplo2}
   S_n := \chi_{[\frac{1}{n+1}, \frac{1}{n}]}.
\end{equation}
Observe that, in particular, $\bigcap_{n=1}^\infty \bigcup_{m=n}^\infty \left[ \frac{1}{m+1}, \frac{1}{m} \right] = \varnothing$. In fact, this property ``characterize'' in some sense those sequences of scalar multiples of indicators functions pointwise convergent to zero but not uniformly convergent. We include the proof to be self-contained. Recall that, for a sequence $(E_n)_n$ of sets, it is defined $$\limsup E_n:=\bigcap_{n\ge 1}\bigcup_{m\ge n}E_m.$$
\begin{proposition} Let $(X, \mathcal{M}, \mu)$ be a measure space, and $(\alpha_n)_n$ be a sequence of nonzero real numbers such that either $(\alpha_n)_n$ converges to $0$ or there exists $M>0$ such that $|\alpha_n|>M$ for $n$ large enough. Let $E_n\in{\mathcal M}\setminus\{\varnothing\}$ and $f_n = \alpha_n \chi_{E_n}$ $(n\in\N)$. Then:
\begin{enumerate}[{\rm (a)}]
\item\label{prop:a} $f_n\to 0$ pointwise on $X$ if and only if $\alpha_n \to 0$ or $\displaystyle \limsup E_n=\varnothing.$
\item\label{prop:0} $f_n\to 0$ a.e.~pointwise on $X$ if and only if $\alpha_n \to 0$ or $\displaystyle \mu\left(\limsup E_n\right)=~0.$
\item\label{prop:b} $f_n\to 0$ uniformly on $X$ if and only if $\alpha_n \to 0$.
\end{enumerate}
\end{proposition}
\begin{proof}\
\begin{enumerate}[(a)]
\item Suppose that $f_n \to 0$ pointwise on $X$ and $\displaystyle \limsup E_n\ne \varnothing$. Let $x_0 \in X$ such that for every $n\in\N$ there exists $m_n\in\N$ with $m_n\ge n$ and $x_0 \in E_{m_n}$. Then
    \begin{equation}\label{alpha}
      |\alpha_{m_n}|=|f_{m_n}(x_0)|\to0 \quad (n\to\infty),
    \end{equation}
     so $(\alpha_n)_n$ must converge to $0$.

Conversely, for any $x\in X$, $|f_n(x)|=|\alpha_n|\chi_{E_n}(x)\le |\alpha_n|$. So, if $\alpha_n\to0$, then $f_n(x)\to 0$. On the other hand, if $\limsup E_n=\varnothing$, then there is $n_0 \in \N $ such that $x \notin E_n$ for all $n\ge n_0$; hence, $f_n(x)=0$ for all $n \geq n_0$ and we are done.

\item Assume that $f_n\to0$ a.e.~in $X$ and $\mu\dis \left(\limsup E_n\right)>0$. Then, we always can find $x_0\in \limsup E_n$ such that $f_n(x_0)$ converges to zero and we can finish as in the previous part.

     For the reciprocal, following the proof of \eqref{prop:a}, it is straightforward that $f_n(x)=0$ for all $n$ large enough and all $x\notin \dis\limsup E_n$, which is a set of measure zero.

\item Suppose that $f_n\to0$ uniformly on $X$. Assume, by way of contradiction, that $\alpha_n\not\to 0$. Then there exist $M>0$ and $n_0 \in \N$ such that $|\alpha_n|>M>0$ for all $n\geq n_0$. But the uniform convergence of $f_n$ allows us to get $m\geq n_0$ such that $|f_n(x)|<\frac{M}{2}$ for all $x\in X$ and $n\ge m$. Therefore, $f_n= 0$ and $E_n=\varnothing$ for all $n\ge m$, which is impossible by hypothesis.

The reciprocal is immediate because of the fact that $|f_n(x)|\le |\alpha_n|$ for all $n\in\N$ and every $x\in X$.\end{enumerate} \end{proof}

\begin{remark}\label{remark:1}
  Observe that similarly as in part \eqref{prop:a} of the above proposition, it can be proved that, given any sequence $(\varphi_n)_n$ of measurable functions, if $\displaystyle \limsup E_n = \varnothing$, then the sequence $\varphi_n\cdot\chi_{E_n}\to0$ pointwise on $X$.
\end{remark}

The above proposition provides us with many sequences of measurable functions with pointwise but not uniform convergence. Let us see that the amount of these sequences in ${L}_0^{\N}$ is huge in both linear and algebraic senses. In fact, we deal with ``stronger'' types of convergence than the pointwise one.

Given $(f_n)_n \in \mathcal{L}_0^\N $ and $f \in \mathcal{L}_0$, we say that $f_n \to f$ almost uniformly on $[0,1]$ if for every $\varepsilon >0$ there exists a set $E\in [0,1]$ with $m(E) \leq \varepsilon$ such that $f_n \to f$ uniformly on $[0,1] \setminus E$; and $f_n \to f$ uniformly a.e.~if there is $E \subset [0,1]$ with $m(E)=0$ such that $f_n \to f $ uniformly on $[0,1] \setminus E$. Uniformly a.e.~convergence can be trivially adapted to $L_0^\N$, but almost uniformly convergence should be slightly adapted to classes of functions. Recall that given $f \in L_0$ and $A \subset [0,1]$, the essential supremum of $f$ at $A$ is defined by
$$ \esssup_{x\in A} f := \inf \{ \alpha \in \R : \ m(\{ x \in A: \ f(x) > \alpha\})=0\}.$$


\begin{definition}\label{defNUP}
A sequence of measurable functions $(f_n:[0,1]\to\R)_n$ is said to belong to the family $NUP([0,1])$, whenever it enjoys the next properties:
\begin{enumerate}[\rm (A)]
\item \label{NUP:A} $f_n \to 0$ pointwise a.e.~in $[0,1]$,
\item \label{NUP:B} for any $\varepsilon >0$ there is a measurable set $E \subset [0,1]$ such that $m(E)< \varepsilon$ and $\esssup_{[0,1] \setminus E} |f_n| \to 0$,
\item \label{NUP:C} $(f_n)_n$ does not converge uniformly a.e.~in $[0,1]$.
\end{enumerate}
\end{definition}

\begin{theorem}\label{PointUnifAlg}
The family $NUP([0,1])$ is strongly $\mathfrak{c}$-algebrable.
\end{theorem}

\begin{proof} As in the proof of Theorem \ref{Teo:1}, let $H\subset(0,+\infty)$ a $\Q$-linearly independent set with ${\rm card}(H)=\mathfrak{c}$. For each $c \in H$ we define the sequence $F(c)=(F(c,n))_n$ by
$$ F(c,n) (x):= e^{-cn[(n+1)x-1]} \cdot S_n (x),$$
where $(S_n)_n$ is the sequence defined in \eqref{ejemplo2}.

Let $\mathcal{B} $ be the algebra generated by the family of sequences $\{ F(c) : \, c \in H\}$. Now, because each $S_n$ is an indicator function, any nonzero member $(F_n)_n$ of $\mathcal{B}$ is of the form \begin{equation}\label{eq:2}F_n(x)=\left(\sum_{{\bf j}\in J}\alpha_{\bf j}e^{-({\bf c}\cdot{\bf j})n[(n+1)x-1]}\right) S_n(x),\end{equation} where, for some $N\in\N$, $J \subset \N_0^N \setminus \{(0,0, \dots ,0)\}$ is a nonempty finite set, $\alpha_{\bf j} \in \R \setminus \{0\}$ for ${\bf j}\in J$ and ${\bf c} \in H^N$.

From \eqref{eq:2}, Remark \ref{remark:1} and the definition of $S_n(x)$, it follows that $F_n(x)$ is pointwise convergent to zero. Moreover, condition \eqref{NUP:B} of Definition \ref{defNUP} also holds because, for every $\eps>0$, $F_n(x)=S_n(x)=0$ for any $n>1/\eps$ and $x\in(\eps,1]$.

Finally, observe that for fixed $n\in\N$, we have that $x\in\left[\frac{1}{n+1},\frac{1}{n}\right]$ if and only if $w:=n[(n+1)x-1]\in [0,1]$. From this fact, together with the definition of $S_n(x)$, we have that
\begin{eqnarray*}
  \dis\esssup_{0\le x\le 1}|F_n(x)| &=& \sup_{\frac{1}{n+1}\le x\le \frac{1}{n}} |F_n(x)| = \sup_{\frac{1}{n+1}\le x\le \frac{1}{n}} \left| \sum_{{\bf j}\in J}\alpha_{\bf j}e^{-({\bf c}\cdot{\bf j})n[(n+1)x-1]}\right|\\
  &=& \dis \sup_{0\le w\le 1}\left|\sum_{{\bf j}\in J}\alpha_{\bf j}e^{-({\bf c}\cdot{\bf j})w}\right|.
\end{eqnarray*}
But this last amount does not depend on $n$ and, in addition, it is positive, because the $\Q$-linearly independence of $H$ implies the linear independence of the set $\{e^{-({\bf c}\cdot{\bf j})x} :\, c\in H,\, {\bf j}\in J\}$. Thus, $(F_n)_n$ does not converge (to zero) uniformly a.e.~in $[0,1]$ and the proof is finished.
\end{proof}

As pointed out in the previous section, Egorov's Theorem guarantees that every member of the algebra $\mathcal{B}$ of the proof of the above theorem, including the sequence $(S_n)_n$ defined in \eqref{ejemplo2}, is also convergent to zero in measure.
So, it is natural to ask about the spaceability of the family of sequences of Theorem \ref{PointUnifAlg}, when $L_0^\N$ is endowed with the product topology inherited from the convergence in measure and its metric defined in Lemma \ref{metrica}.

\begin{theorem}\label{spaceability2}
  The family $NUP([0,1])$ is spaceable in $L_0^\N$.
\end{theorem}

\begin{proof}
Let $E_n:=\left[\frac{1}{n+1},\frac{1}{n}\right]$. Let us divide $\N$ into infinitely many pairwise disjoint subsequences $\{(i(k,n))_n:\, k\in\N\}$, such that $i(k,n)< i(k',n)$ for $k< k'$ (again, as in the proof of Theorem \ref{spaceability1}, $i(k,n):=k(k+1)/2+(n-1)k$ do the job). For each $k\in\N$, define the sequence $S(k)=(S(k,n))_n:=(\chi_{E_{i(k,n)}})_n$.

First of all, we are going to prove that $\{S(k):\ k\in\N\}$ is a linearly independent set. Indeed, let $\lambda_1,\ldots,\lambda_N\in\R$ and pairwise different $k_1,\ldots,k_N\in\N$ such that $\sum_{j=1}^N\lambda_j S(k_j)$ is the null sequence. Then, for every $n\in\N$ and every $x\in[0,1]$, we have \begin{equation}\label{independence}\lambda_1\chi_{E_{i(k_N,n)}}(x)+\ldots+\lambda_N\chi_{E_{i(k_N,n)}}(x)=0.\end{equation}
But, by construction, if $(k,n)\ne(k',n')$ then $i(k,n)\ne i(k',n')$, so $E_{i(k,n)}\cap E_{i(k',n')}$ is either empty or a singleton. Then, for $1\le j\le N$ we always can find $$x_j\in  E_{i(k_j,n)}\setminus \bigcap_{\substack{1\le \nu \le N\\ \nu\ne j}} E_{i(k_{\nu},n)}$$ and applying \eqref{independence} at $x=x_j$, we get that $\lambda_j=0$ for $1\le j\le N$.

Let $M:=\overline{\rm span}\{S(k):\ k\in\N\}$. It is clear that $M$ is an infinite dimensional closed subspace of $L_0^\N$. We claim that every nonzero member of $M$ enjoys properties \eqref{NUP:A}, \eqref{NUP:B} and \eqref{NUP:C} of Definition \ref{defNUP}.

Given $F=(F_n)_n\in M\setminus\{0\}$, there exists a strictly increasing sequence $(k_j)_j\subset\N$ and a sequence $(\alpha_j)_j\subset\R$ (not identically zero), such that $F=\sum_{j=1}^\infty \alpha_jS(k_j)$ and, without loss of generality, we may assume that $\alpha_1\ne0$.

For every $j,n\in\N$, it is clear that $0<\frac{1}{i(k_j,n)}<\frac{1}{i(k_1,n)}\to0$ ($n\to\infty$). So, for any $x\in[0,1]$, there is a number $n_0\in\N$ such that $F_n(x)=0$ for all $n\ge n_0$. Hence, $F_n$ is convergent to 0 in $[0,1]$ and we have \eqref{NUP:A}. Moreover, given $\eps>0$ there is $n_1\in\N$ such that $E_{i(k_j,n)}\subset[0,\eps/2]$ for every $n\ge n_1$ and every $j\in\N$. Hence, $F_n(x)=0$ for every $n>n_1$ and every $x\in(\eps/2,1]$, and we get \eqref{NUP:B}.

Finally, \eqref{NUP:C} holds because, for every $n$ of the form $2^{i(k_1,m)}+j$ for some $m\in\N$ and $0\le j<2^{i(k_1,m)}$, we have that
$$\esssup_{0\le x\le 1}|F_n(x)|\ge \esssup_{x\in E_{i(k_1,m)}} |F_n(x)| = \sup_{x\in E_{i(k_1,m)}} |\alpha_1 S(k_1,n)(x)|=|\alpha_1|\ne0.$$
\end{proof}

\begin{remark}\label{Remark:KandT}
 Observe that, following \cite{bernalordonez2014}, the above theorem (and Theorem \ref{spaceability1}) can also be proved by using a general result on spaceability of Kitson and Timoney (see \cite[Theorem 2.2]{KitsonTimoney} or \cite[\S7.4]{aronbernalpellegrinoseoane}):
 \begin{quotation}{\em
   If $Y$ is a closed vector subspace of a Fr\'echet space $X$, then $X\setminus Y$ is spaceable if and only if $Y$ has infinite codimension.}
 \end{quotation}
    But this technique, at least in the above results, does not shorten the proofs and make them less constructive.
\end{remark}

It is a direct consequence of both Theorems \ref{PointUnifAlg} and \ref{spaceability2} that the family $NUP([0,1])$ is $\mathfrak{c}$-lineable. But again bringing up the topological structure of $L_0^\N$, and taking into account that any complete separable metric topological vector space has dimension at most $\mathfrak{c}$, we can prove also the maximal dense-lineability of this family.

Prior to this, we need an auxiliary, general result about lineability. The following result, being one of the many existing variants, can be found in \cite[Theorem 2.3]{bernalordonez2014} (see also \cite[Theorem 2.2 and Remark 2.5]{topology} and \cite[Section 7.3]{aronbernalpellegrinoseoane}).

\begin{lemma} \label{dense-lineability}
Assume that $Z$ is a metrizable, separable, topological vector space and that $\gamma$ is an infinite cardinal number.
Suppose that $A$ and $B$ are subsets of $Z$ such
that $A + B\subset A$, $A \cap B = \varnothing$, $B$ is dense-lineable and $A$ is $\gamma$-lineable.
Then $A$ is $\gamma$-dense-lineable.
\end{lemma}

\begin{theorem}\label{maxdenslin2}
The family $NUP([0,1])$ is maximal dense-lineable.
\end{theorem}
\begin{proof}Consider the set $L_{00}$ of eventually null sequences of functions in $L_0$, that is,
$$L_{00} := \{ (f_n)_n \in L_0^\N : \ \text{there exists } N \in \N  \text{ such that } f_n = 0 \text{ for } n \geq N\}.$$
As $L_0^\N $ is endowed with the product topology it is obvious that $L_{00} $ is dense vector subspace of $L_0^\N$, hence dense-lineable.

Let $\mathcal{A}$ be the family of the hypothesis. We know that $\mathcal{A}$ is $\mathfrak{c}$-lineable. Moreover, $\mathcal{A}\cap L_{00}=\varnothing$, because every member of $L_{00}$ is convergent uniformly, and $\mathcal{A}+L_{00}\subset A$.

Now, an application of Lemma \ref{dense-lineability} with $Z=L_0^\N$ (recall that $\dim(L_0^\N)=\mathfrak{c}$), $A=\mathcal{A}$, $B=L_{00}$ and $\gamma =\mathfrak{c}$, finishes the proof.
\end{proof}

\section{Uniform vs $\|\cdot\|_{L^1}$-norm convergence}

In this final section we focus on sequences of measurable functions uniformly convergent to 0 but not in $L^1$-norm. Unlike what happened in the previous sections, this phenomenon cannot happen in a finite measure setting. So we will work with measurable functions defined in $[0,+\infty)$.

For every $n\in\N$, let $R_n:[0,+\infty)\to\R$ given by
\begin{equation}\label{ejemplo3}
R_n := \frac{1}{n} \chi_{[0,n]}.\end{equation}
The sequence $(R_n)_n$ is a classical example of sequence of functions converging uniformly to $0$ but not in $L^1$-norm and is the germ of the proof of all results in this section.

From now on, let $Z_0:=\left(L_0[0,+\infty)\right)^\N$. By $\|f\|_1$ we mean the $L^1$-norm of a measurable function $f:[0,+\infty)\to\R$, that is, $$\|f\|_1=\int_0^{+\infty}|f(x)|dx.$$

\begin{theorem}\label{algebrability3}
The family of sequences $(f_n)_n \in Z_0$ such that $(f_n)_n$ is uniformly convergent to 0 but not in $L^1$-norm, is strongly $\mathfrak{c}$-algebrable.
\end{theorem}

\begin{proof}Let $H\subset(0,+\infty)$ be a $\Q$-linearly independent set with $\dim(H)=\mathfrak{c}$. For every $c\in H$, let $F(c):=(F(c,n))_n\in Z_0$ be the sequence given by
\begin{equation}\label{defini3}
F(c,n):=\frac{1}{n^c}\cdot\chi_{[0,e^n]}.
\end{equation}
  It is clear that $|F(c,n)(x)|\le \frac{1}{n^c}\to0$. In addition, $\|F(c,n)\|_1=\frac{e^n}{n^c}\to+\infty$ for every $c\in H$. So for any $c\in H$, the sequence $F(c)$ is uniformly convergent to 0 in $[0,+\infty)$ but not in $L^1$-norm.

  Let $\mathcal{B}$ be the algebra generated by the family $\{F(c):\, c\in H\}$, which is clearly a linearly independent family. We claim that $\mathcal{B}$ is a free algebra such that any non-zero member is a sequence uniformly convergent to 0 in $[0,+\infty)$ but not in $L^1$-norm.

  Let $F=(F_n)_n\in\mathcal{B}\setminus\{0\}$. Similarly to the proof of Theorems \ref{Teo:1} and \ref{PointUnifAlg}, there exist a natural number $N\in\N$, a finite set $J\subset\N^N\setminus\{(0,\overset{(N)}{\ldots},0)\}$ and scalars $\alpha_{\bf j}\in\R\setminus\{0\}$ for any ${\bf j}\in J$ such that, for every $n\in\N$, we have \begin{equation}\label{eq:3}F_n=\left(\sum_{{\bf j}\in J}\alpha_{\bf j} \frac{1}{n^{{\bf c}\cdot{\bf j}}}\right)\chi_{[0,e^n)}.\end{equation}

  The $\Q$-linearly independence of $H$ and \eqref{eq:3}, guarantee that, for any $n\in\N$, the real number $\sum_{{\bf j}\in J}\alpha_{\bf j} \frac{1}{n^{{\bf c}\cdot{\bf j}}}$ can never be zero and, in addition, $\frac{1}{n^{{\bf c}\cdot{\bf j}}}\to0$. So, the algebra $\mathcal{B}$ is free, the sequence $F(c)=(F(c,n))_n$ is uniformly convergent to 0 in $[0,+\infty)$ and $$\|F(c,n)\|_1=\left|\sum_{{\bf j\in J}}\alpha_{\bf j}\frac{1}{n^{{\bf c}\cdot{\bf j}}}\right|e^n\to+\infty,\qquad (n\to\infty)$$ which concludes the proof.
\end{proof}

Although every non-zero sequence of the algebra $\mathcal{B}$ of the previous proof is not convergent (to zero) in $L^1$-norm, all the functions are, in fact, integrable in $[0,+\infty)$. Hence $\mathcal{B}\setminus\{0\}\subset Z_1:=\left(L^1[0,+\infty)\right)^\N$, where $L^1[0,+\infty)$ denotes the set of (classes of) integrable functions in $[0,+\infty)$. This allows us to endow $Z_1$ with the product topology inherited by the $\L^1$ norm, thus making $Z_1$ a separable, metrizable, topological vector space.

As stated in Section 1, Bernal and Ord\'o\~nez \cite[\S 4.7]{bernalordonez2014} considered the space $CBL_s$ of sequences of continuous, bounded and integrable functions $f_n:\R\to\R$ such that $\|f_n\|_\infty\to0$ and $\sup_n\|f_n\|_1+\infty$ (which becomes a Banach space with the norm $\|(f_n)_n\|:=\sup_n\|f_n\|_\infty+\sup_n\|f_n\|_1$). They proved that the family $\mathcal{F}:=\{(f_n)_n\in CBL_s:\, \|f_n\|_1\not\to0 \text{ as }n\to\infty\}$ is spaceable in $CBL_s$. It turns out that the family $\mathcal{F}$ is smaller than the family of Theorem \ref{algebrability3}, but the topology of $CBL_s$ is finer than those of $Z_1$. Thus, it cannot be directly derived from \cite{bernalordonez2014}, the spaceability of the family of functions of Theorem \ref{algebrability3}. Moreover, in \cite{bernalordonez2014} it is also showed that the Banach space $CBL_s$ is not separable, so they were not able to state the maximal dense-lineability. In this setting, taking into account that the product topology makes $Z_1$ separable, we are able to prove the next result.
%
%

\begin{theorem}\label{maxdenslin3}
The family of sequences $(f_n)_n \in Z_1$ such that $(f_n)_n$ is uniformly convergent to 0 but not in $L^1$-norm, is maximal dense-lineable and spaceable in $Z_1$.
\end{theorem}

\begin{proof}
Both from Bernal and Ord\'o\~nez result or directly from Theorem \ref{algebrability3} (at this moment we does not care about the topology), we know that the family $\mathcal{A}$ of the hypothesis is $\mathfrak{c}$-lineable. Additionally, $Z_1$ is a separable complete topological vector space (so its dimension is $\mathfrak{c}$) and we can use Lemma \ref{dense-lineability} with $Z=Z_1$, $A=\mathcal{A}$ and $B=L^1_{00}$ the set of eventually null sequences with nonzero components in $L^1[0,+\infty)$. Observe that, similarly to $L_{00}$ in the proof of Theorem \ref{maxdenslin2}, this set is a dense vector space of $Z_1$ and has empty intersection with $\mathcal{A}$, because every sequence in $L^1_{00}$ is convergent to zero in $L^1$-norm.

It remains to prove the spaceability. For this purpose, let us divide the interval $[0,+\infty)$ into infinitely many sequences of pairwise disjoint intervals (except, possibly, for the extremes). For every $N\in\N$ and every $M=1,\ldots,N$, let
$$I_{N,M}:=\left[\sum_{j=1}^{N-1}j(N-j)+\frac{M(M-1)}{2}, \sum_{j=1}^{N-1}j(N-j)+\frac{M(M+1)}{2}\right].$$
Observe that for each $M\in\N$, the interval $I_{N,M}$ has always length $M$.

For every $k\in\N$, define the sequence $G(k)=(G(k,n))_n:=\left(\frac{1}{n}\chi_{I_{k+n-1,n}}\right)_n$. It is straightforward that every sequence $G(k)$ converges to zero uniformly in $[0,+\infty)$ but not in $L^1$-norm (observe that $\|G(k,n)\|_1=1$ for all $n\in\N$). Moreover, the family $\{G(k):\,k\in\N\}$ is linearly independent because of the disjointness of the (interiors of the) supports of all functions included in it.

Let $X:=\{(f_n)_n\in Z_1:\, f_n\to0\, \text{uniformly in } [0,+\infty)\}$ (which becomes a Fréchet space when endowed with the product topology inherited by the $L^1$ norm) and $Y:=\{(f_n)_n\in X:\, \|f_n\|_1\to0\}$, which is trivially a closed subspace of $X$. By construction, we have that $\{G(k):\, k\in\N\}\subset X\setminus Y$, whence $Y$ has infinite codimension. Now a direct application of Kitson and Timoney criterion for spaceability (see Remark \ref{Remark:KandT}), gives us the spaceability of $\mathcal{A}$ in $X$, and hence in $Z_1$.
%
\end{proof}

\begin{remark} All the sequences of Theorems \ref{algebrability3} are not convergent (to zero) in $L^1$-norm. In fact, by construction, given any sequence $(f_n)_n$ from the algebra of Theorem \ref{algebrability3},
it holds that $\|f_n\|_1\to+\infty$ ($n\to\infty$). However, every sequence of the closed vector space given by the result of Bernal and Ord\'o\~nez is uniformly bounded in $L^1$-norm. Thus, it is natural to ask about the algebraic genericity of the family of sequences of functions $(f_n)_n\in Z_1$ such that $\sup_n\|f_n\|_1<+\infty$, $f_n\to0$ uniformly in $[0,+\infty)$ but not in $L^1$-norm. Similarly to \cite[Theorem 4.16]{bernalordonez2014}, it can be showed the spaceability of this family in $Z_1$. Moreover, a direct application of Lemma \ref{dense-lineability} with $\gamma=\mathfrak{c}$ and $B=L^1_{00}$, allows us to prove the maximal-dense-lineability.
\end{remark}

In view of the last remark, we want to finish this paper by posing the next open problem.

\begin{problem}
Consider the family of sequences of functions $(f_n)_n\in (L^1[0,+\infty))^\N$ such that $\sup_n\|f_n\|_1<+\infty$, $(f_n)_n$ is uniformly convergent to 0 but not in $L^1$-norm. Is this family strongly $\mathfrak{c}$-algebrable?
\end{problem}

\bigskip

\bigskip

{\scriptsize
\begin{flushleft}
M.C.~Calder\'on-Moreno, P.J.~Gerlach-Mena and J.A.~Prado-Bassas\\
Departamento de An\'alisis Matem\'atico.  \\
Facultad de Matem\'aticas, Universidad de Sevilla. \\
Avda.~Reina Mercedes s/n, 41012 Sevilla, Spain.  \\
E-mails: {\tt mccm@us.es}, {\tt gerlach@us.es} and {\tt bassas@us.es}
\end{flushleft}}
\end{document}